\documentclass[11pt,reqno]{amsart}
\usepackage{a4wide}
\usepackage[utf8x]{inputenc}
\usepackage[T1]{fontenc}
\usepackage[english]{babel}
\usepackage{hyperref}
\usepackage{amsfonts,amsmath,amssymb,amsthm,amsrefs}
\usepackage{latexsym}
\usepackage{layout}
\usepackage{dsfont}
\usepackage{xcolor}
\usepackage{soul}
\usepackage{graphicx}

\providecommand{\R}{}

\renewcommand{\R}{\mathbb{R}}

\newcommand{\I}[1]{{\mathbf 1}_{\left\{#1\right\}}}									
\newcommand{\set}[1]{\left\{ #1 \right\}}								

\newcommand{\esub}[1]{{\mathbf E_{#1}}}

\newcommand\cF{\mathcal F}

\newcommand\cS{{\mathcal S}}

\newcommand{\bbW}{\mathbb{W}}
\newcommand{\dd}{\mathrm{d}}

\providecommand{\ora}[1]{}
\renewcommand{\ora}[1]{\overrightarrow{#1}}

\newcommand{\threepartdef}[6]
{
	\left\{
		\begin{array}{ll}
			#1 & \mbox{if } #2 \\
			#3 & \mbox{if } #4 \\ 
			#5 & \mbox{if } #6
		\end{array}
	\right.
}
\newcommand{\twopartdef}[4]
{
	\left\{
		\begin{array}{ll}
			#1 & \mbox{if } #2 \\
			#3 & \mbox{if } #4
		\end{array}
	\right.
}
\newtheorem{thm}{Theorem}
\newtheorem{lem}[thm]{Lemma}
\newtheorem{prop}[thm]{Proposition}
\newtheorem{cor}[thm]{Corollary}
\newtheorem{dfn}[thm]{Definition}

\newtheorem{rem}{Remark}
\numberwithin{thm}{section}
\numberwithin{equation}{section}
\begin{document}

\title[Optimality of a refraction strategy subject to {P}arisian ruin]{Optimality of a refraction strategy in the optimal dividends problem with absolutely continuous controls subject to {P}arisian ruin}

\author[Locas]{F\'elix Locas}
\address{D\'epartement de math\'ematiques, Universit\'e du Qu\'ebec \`a Montr\'eal (UQAM), 201 av.\ Pr\'esident-Kennedy, Montr\'eal (Qu\'ebec) H2X 3Y7, Canada}
\email{locas.felix@uqam.ca, renaud.jf@uqam.ca}

\author[Renaud]{Jean-Fran\c cois Renaud}


\date{\today}

\keywords{Dividend payments, absolutely continuous strategies, Parisian ruin, spectrally negative L\'evy process, refraction strategies}

\begin{abstract}
We consider de Finetti's optimal dividends problem with absolutely continuous strategies in a spectrally negative L\'evy model with Parisian ruin as the termination time. The problem considered is essentially a generalization of both the control problems considered by Kyprianou, Loeffen \& P\'erez \cite{kyprianou-et-al_2012} and by Renaud \cite{renaud_2019}. Using the language of scale functions for Parisian fluctuation theory, and under the assumption that the density of the L\'evy measure is completely monotone, we prove that a refraction dividend strategy is optimal and we characterize the optimal threshold. In particular, we study the effect of the rate of Parisian implementation delays on this optimal threshold.
\end{abstract}

\maketitle


\section{Introduction}\label{introsection}

In 2007, Avram, Palmowski \& Pistorius \cite{avram-palmowski-pistorius_2007} kickstarted a string of literature concerned with Bruno de Finetti's control problem \cite{definetti_1957} for spectrally negative Lévy processes (SNLPs). Using results from fluctuation theory for SNLPs, they studied this optimal dividends problem consisting in finding the optimal strategy maximizing the withdrawals made up to ruin. In particular, and inspired by results obtained in simpler models, they wanted to answer the following question: when is a barrier strategy optimal? In a pioneering follow-up work, Loeffen \cite{loeffen_2008} gave a clear and satisfactory sufficient condition on the Lévy measure for a barrier strategy to be optimal. In words, if the Lévy measure admits a completely monotone density, then a barrier strategy is optimal. Later, this condition was improved first by Kyprianou, Rivero \& Song \cite{kyprianou-et-al_2010} and then further improved by Loeffen \& Renaud \cite{loeffen-renaud_2010}. To the best of our knowledge, the condition used in \cite{loeffen-renaud_2010} is the mildest condition known to date; it says that, if the tail of the Lévy measure is a log-convex function, then a barrier strategy is optimal.

The stochastic control problem discussed in the previous paragraph is a singular control problem. It is one of the three versions of this classical control problem, the other two being the \textit{impulse problem} and the \textit{absolutely continuous problem}. See \cite{jeanblanc-shiryaev_1995} for a study of these three problems in a Brownian model. In the impulse problem, the analog of a barrier strategy is an $(a,b)$-strategy. It is interesting to note that again, if the tail of the Lévy measure is a log-convex function, then a certain $(a,b)$-strategy is optimal; see in sequence \cite{loeffen_2009}, \cite{loeffen-renaud_2010} and \cite{renaud_2024} for details and discussions on this matter. The third classical variation consists in restricting the set of admissible controls to absolutely continuous (with respect to the Lebesgue measure) strategies. In that case, the question becomes: when is a refraction strategy optimal? This problem has been studied, in a spectrally negative Lévy model, by Kyprianou, Loeffen \& Pérez \cite{kyprianou-et-al_2012} and the strongest condition of having a Lévy measure with a completely monotone density was again needed for a refraction strategy to be optimal. Once more, to the best of our knowledge, no one has been able yet to improve this condition.

In the last fifteen years or so, fluctuation theory for SNLPs was further developed by the addition of so-called Parisian ruin identities; see for example \cites{landriault-et-al_2011, loeffen-et-al_2013,albrecher-et-al_2016, lkabous-renaud_2019}. Simply said, in Parisian fluctuation theory, a barrier is said to have been crossed once it has been crossed for a given amount of time, known as a Parisian implementation delay. Naturally, studies have started to appear on the impact of these delays, used in the recognition of ruin, on the maximization of dividends, especially on the optimality of a barrier strategy. The singular version of this problem has been studied in \cite{czarna-palmowski_2014} and in \cite{renaud_2019} using a Parisian termination time with deterministic delays and exponential delays, respectively. Very recently, in \cite{renaud_2024}, the impulse control problem with Parisian exponential ruin has been considered. When using exponential Parisian delays, the condition of log-convexity of the tail of the Lévy measure was used as a sufficient condition for a barrier strategy (resp.\ an $(a,b)$-strategy) to be optimal in the singular problem (resp.\ impulse problem).

Finally, note that, from a modelling perspective, using exponential delays in the definition of ruin could be seen as an \textit{approximation} of deterministic delays. This modelling assumption is compensated by more explicit (analytical) results allowing for additional understanding of the model and the control problem. Note also that this choice of delays makes it equivalent to being able to declare bankruptcy only at Poisson times, thanks to the memoryless property of the exponential distribution. See, e.g., \cite{albrecher-cheung-thonhauser_2013}.

In this paper, we analyze de Finetti's control problem for absolutely continuous strategies in a SNLP model with exponential Parisian ruin. To solve this problem, we walk in the footsteps of \cite{kyprianou-et-al_2012} and \cite{renaud_2019} by using (and generalizing) methodologies and results used in those two papers. More precisely, we will show that a threshold strategy is optimal if the Lévy measure of the underlying SNLP admits a completely monotone density.

\subsection{Model, control problem and main result}\label{mathematical-setup-section}

We consider a model in which the uncontrolled surplus process $X = \set{X_t : t \geq 0}$ is a spectrally negative Lévy process (SNLP), i.e., $X$ is a Lévy process with no positive jumps. More details on SNLPs are provided in Section~\ref{preliminarysection}.

For a given control $\pi$, characterized by a process $L^\pi = \set{L_t^\pi : t \geq 0}$ where $L^\pi_t$ is the cumulative amount of dividends paid up to time $t$, the controlled process is given by $U^\pi = \set{U_t^\pi : t \geq 0}$ and defined by $U_t^\pi = X_t - L_t^\pi$.

Let us now define the Parisian termination time used in our model. First, fix a Parisian rate $p > 0$. Then, the ruin time corresponding to a strategy $\pi$ is the following Parisian first-passage time:
\begin{equation}\label{def_parisian-ruin}
\kappa_p^\pi := \inf\set{t > 0 : t - g_t^\pi > \mathbf{e}_p^{g_t^\pi} \text{ and } U_t^\pi < 0},
\end{equation}
where $g_t^\pi := \sup\set{0 \leq s \leq t : U_s^\pi \geq 0}$, and where $\mathbf{e}_p^{g_t^\pi}$ is an exponential random variable with mean $1/p$ and independent of the sigma-algebra $\bigvee_{t \geq 0} \cF_t$. Note that there is a \textit{new} exponential random variable associated to each excursion (of $U^\pi$) below zero.

As mentioned above, we are interested in absolutely continuous strategies. Therefore, we will consider strategies $\pi$ such that
\begin{equation*}
L_t^\pi = \int_0^t l_s^\pi \mathrm{d}s , \quad t \geq 0,
\end{equation*}
where $\set{l_t^\pi : t \geq 0}$ is an adapted and nonnegative stochastic process. Note that, as a consequence, we have that $\set{L_t^\pi : t \geq 0}$ is an adapted and nondecreasing stochastic process such that $L^\pi_0=0$. Let us fix a maximal dividend rate $K > 0$.  A strategy $\pi$ is said to be admissible if it is such that $0 \leq l_t^\pi \leq K$, for all $0 \leq t \leq \kappa_p^\pi$, and if $t \mapsto l_t^\pi \mathbf{1}_{\set{U_t^\pi < 0}} \equiv 0$, which means that no dividends are paid when the controlled process is negative. Let us denote by $\Pi_K$ the set of these admissible strategies.


Finally, let us fix a discount factor $q>0$ and let the performance function of a strategy $\pi$ be given by
\begin{equation}
V_\pi(x) = \esub{x}\left[\int_0^{\kappa_p^\pi} e^{-qt} \mathrm{d}L_t^\pi\right] = \esub{x}\left[\int_0^{\kappa_p^\pi} e^{-qt} l^\pi_t \mathrm{d}t \right] , \quad x \in \R .
\label{definition-performance-function}
\end{equation}
Consequently, the value function of this problem is
\[
V(x) = \sup_{\pi \in \Pi_K} V_\pi(x), \quad x \in \R .
\]
We want to find an optimal strategy $\pi^* \in \Pi_K$, that is a dividend strategy such that $V_{\pi^*}(x) = V(x)$ for all $x \in \R$.

In order to state our main result, we need to introduce the sub-family of refraction strategies. For a given $b \geq 0$, let $\pi_b$ be the refraction strategy at level $b$, i.e., the strategy paying dividends at the maximum rate $K$ when the (controlled) surplus process is above level $b$ and not paying dividends when the surplus process is below level $b$. More precisely, for this strategy, we write
\begin{equation}\label{edsrefracted}
\mathrm{d}U_t^b = \mathrm{d}X_t - K\mathbf{1}_{\set{U_t^b > b}} \mathrm{d}t ,
\end{equation}
for the controlled process, from which we deduce the dividend rate $l^b_t = K \mathbf{1}_{\set{U_t^b > b}}$, and we write $\kappa_p^b$ instead of $\kappa_p^{\pi_b}$ for the ruin time. Note that, for any $b \geq 0$, the refraction strategy at level $b$ is admissible, i.e., $\pi_b \in \Pi_K$.


We are now ready to state the main result of this paper.
\begin{thm}\label{mainresult}
If the L\'evy measure of $X$ has a completely monotone density, then there exists an optimal threshold $b^* \geq 0$ such that $\pi_{b^*}$ is an optimal strategy.
\end{thm}

In fact, this is a preliminary version of our full solution to this control problem. In Theorem~\ref{mainresult2}, we provide an explicit expression for the value function $V$, including a characterization of the optimal threshold $b^*$.

\subsection{Contributions and outline of the paper}

The control problem described above can be considered as a generalization of the problem studied in \cite{kyprianou-et-al_2012} and the one studied in \cite{renaud_2019}. Indeed, in principle, if $p \to \infty$, then we should recover the problem (and results) in \cite{kyprianou-et-al_2012}, while if $K \to \infty$, then we should recover the problem in \cite{renaud_2019}; more on this at the end of Section~\ref{refractionsection}.

To prove Theorem~\ref{mainresult}, we combine ideas from the abovementioned two papers, using a now \textit{standard} methodology for optimal dividends problems in spectrally negative Lévy models. More precisely, we compute first the performance function of an arbitrary refraction strategy. Second, we find the optimal refraction level $b^* \geq 0$, i.e., identify the \textit{best} refraction strategy $\pi_{b^*}$ within the sub-family of refraction strategies. Third, using a verification lemma, we prove that this refraction strategy $\pi_{b^*}$ is in fact an optimal admissible strategy for our control problem. Even if we use the same problem-solving methodology, there are several technical difficulties that arise in our problem, which lead to new results. For example:
\begin{itemize}
\item in Proposition~\ref{performance-proposition}, we compute the performance function of an arbitrary refraction strategy subject to Parisian ruin, which is a function defined on $\R$, and to do so a new technical result is needed (see Lemma~\ref{new-identity-lemma});

\item in Lemma~\ref{caracteristiques}, when characterizing the optimal refraction level, we derive important analytical properties for a particular function (defined in Equation~\eqref{definitionh}), which go further than the case with classical ruin;

\item the proof of Proposition~\ref{hardest} is more involved than the corresponding one in \cite{kyprianou-et-al_2012} and, furthermore, we provide an alternative proof in the case $b^* = 0$.
\end{itemize}

The rest of the paper is organized as follows. In the next section, we present mathematical objects and some technical results needed to solve our control problem. Then, in Section~\ref{refractionsection}, we compute the performance function of an arbitrary refraction strategy and identify the optimal refraction level. Also, we relate our results to those obtained in \cite{kyprianou-et-al_2012} and in \cite{renaud_2019}. Finally, in Section~\ref{verificationsection}, we provide a verification lemma for the control problem and then use it to finish the proofs of Theorem~\ref{mainresult} and Theorem~\ref{mainresult2}.

\section{Preliminary results}\label{preliminarysection}

Now, let us be a little more specific about our model and present some of the mathematical objects needed for our solution.

First, some terminology and notation. For a function $f \colon \R \to \R$, we write $f(\infty) := \lim_{x \to \infty} f(x)$ if this limit exists. Also, we say that $f$ is completely monotone if $(-1)^n f^{(n)}(x) \geq 0$, for all $x$, where $f^{(n)}$ stands for the $n$-th derivative of $f$, if $n \geq 1$, and where $f^{(0)}=f$.  Implicitly, a completely monotone function $f$ is nonnegative and infinitely differentiable. Finally, $f$ is said to be log-convex if $\log(f)$ is convex.

\subsection{Spectrally negative Lévy processes}

Let $(\Omega, \cF, \left(\cF_t\right)_{t \geq 0}, \mathbf{P})$ be a filtered probability space on which $X = \set{X_t : t \geq 0}$ is a spectrally negative Lévy process (SNLP) with L\'evy triplet $(\gamma, \sigma, \nu)$, where $\gamma \in \R$, $\sigma \geq 0$ and where the Lévy measure $\nu$ is such that $\nu(-\infty, 0) = 0$ and
\begin{equation*}
\int_{(0, \infty)} (1 \wedge z^2) \nu(\mathrm{d}z) < \infty .
\end{equation*}
Recall that, for an SNLP, we have the existence of a Laplace exponent given by
\begin{equation*}
\psi(\lambda) = \gamma \lambda + \frac{1}{2}\sigma^2 \lambda^2 - \int_{(0, \infty)} (1 - e^{-\lambda z} - \lambda z \I{0 < z \leq 1}) \nu(\mathrm{d}z) , \quad \lambda \geq 0 .
\end{equation*}
As $X$ is a strong Markov process, we denote by $\mathbf{P}_x$ the law of $X$ when starting from $X_0 = x$ and by $\mathbf{E}_x$ the corresponding expectation. When $x = 0$, we write $\mathbf{P}$ and $\mathbf{E}$, as in the definition of the Laplace exponent above. It is well known that $\psi$ is a strictly convex function such that $\psi(0) = 0$ and $\psi(\infty) := \lim_{\lambda \to \infty} \psi(\lambda) = \infty$, with right-inverse function given by
\begin{equation*}
\Phi(q) = \sup\set{\lambda \geq 0 : \psi(\lambda) = q}.
\end{equation*} 

Recall also that if $\nu \equiv 0$, then $X$ is a Brownian motion with drift, while if $\sigma=0$ and $\nu(0, \infty) < \infty$, then $X$ is a compound Poisson process with drift. In this direction, we know that $X$ has paths of bounded variation (BV) if and only if $\sigma = 0$ and $\int_0^1 z \nu(\mathrm{d}z) < \infty$; otherwise, $X$ is said to have paths of unbounded variation (UBV). More precisely, if $X$ has paths of BV, then we can write
\begin{equation*}
X_t = ct - S_t ,
\end{equation*}
where $c := \gamma +\int_0^1 z \nu(\mathrm{d}z)>0$ and $S = \set{S_t : t \geq 0}$ is a driftless subordinator. 

For our fixed value of $K > 0$, we define another SNLP $Y = \set{Y_t : t \geq 0}$ by $Y_t = X_t - Kt$. It is easy to see that the Laplace exponent of $Y$ is given by $\psi_K(\lambda) = \psi(\lambda) - K \lambda$. Accordingly, we denote its right-inverse by $\Phi_K$.

As a standing assumption throughout this paper, we assume that, if $X$ has paths of BV, then $K \in (0,c)$. This is to avoid that the paths of $Y$, which are also of BV, be monotone decreasing. 

Finally, recall that, for any $\beta > 0$, it is known that the process $\mathcal{E}(\beta) = \set{\mathcal{E}_t(\beta) : t \geq 0}$, defined as
\begin{equation*}
\mathcal{E}_t(\beta) = \exp \left\lbrace \beta X_t - \psi(\beta) t \right\rbrace ,
\end{equation*}
is a unit-mean $\mathbf{P}$-martingale with respect to the filtration $(\mathcal{F}_t)_{t \geq 0}$. Hence, it may be used to define the following Esscher transformation:
\begin{equation}
\left.\frac{\mathrm{d}\mathbf{P}^\beta}{\mathrm{d}\mathbf{P}}\right|_{\mathcal{F}_t} = \mathcal{E}_t(\beta), \quad t \geq 0.
\label{esscher}
\end{equation}
Under the measure $\mathbf{P}^\beta$, it can be shown that the process $(X, \mathbf{P}^\beta)$ is also a SNLP.

For more details on spectrally negative L\'evy processes, we refer the reader to \cite{kyprianou_2014}. 


\subsection{Scale functions and fluctuation identities}

For any $a \in \R$, we define the following first-passage times:
\[
\tau_a^- = \inf\set{t ­> 0 : X_t < a} \quad \text{and} \quad \tau_a^+ = \inf\set{t > 0 : X_t > a} .
\]
Correspondingly, for $Y$, we will write $\nu_a^-$ and $\nu_a^+$. Also, let $\kappa_p$ be the Parisian ruin time of $X$, i.e., the random time defined in~\eqref{def_parisian-ruin} with the (null) strategy $\hat{\pi}$ such that $L^{\hat{\pi}}_t = 0$ for all $t \geq 0$.

To study those first-passage times, we will need different families of scale functions. First, for $q \geq 0$, the $q$-scale function of $X$ is defined as the continuous function on $[0,\infty )$ with Laplace transform 
\[
\int_{0}^{\infty} e^{-\theta y} W^{(q)}(y) \mathrm{d}y = \frac{1}{\psi(\theta)-q} , \quad \text{for $\theta>\Phi(q)$.}
\]
From this definition, we can show that
\begin{equation}\label{Wen0}
W^{(q)}(0+) = \twopartdef{1/c,}{\text{$X$ has paths of BV,}}{0,}{\text{$X$ has paths of UBV.}}
\end{equation}
It is known that this function is positive, strictly increasing and differentiable almost everywhere on $(0, \infty)$. We extend $W^{(q)}$ to the whole real line by setting $W^{(q)}(x)=0$ for $x<0$. We will write $W=W^{(0)}$ when $q=0$. 

Second, for $q,\theta \geq 0$, set
\[
Z_q (x,\theta) = e^{\theta x} \left( 1 - (\psi (\theta)-q) \int_0^x e^{-\theta y} W^{(q)}(y) \mathrm{d}y \right) , \quad x \in \R.
\]
Note that, for $x \leq 0$, we have $Z_q (x,\theta)=e^{\theta x}$. In the following, we will use these functions only with $\theta = \Phi(p+q)$ and $\theta = 0$. Consequently, let us define
\begin{equation}\label{Zphi(p+q)}
Z_{q,p}(x) := Z_q(x, \Phi(p+q)) = e^{\Phi(p+q) x} \left(1 - p \int_0^x e^{-\Phi(p+q) y} W^{(q)}(y) \mathrm{d}y \right) ,
\end{equation}
and $Z^{(q)}(x) := Z_q(x, 0) = 1 + q \int_0^x W^{(q)}(y) \mathrm{d}y$. Similarly, let $\mathbb{W}^{(q)}$ and $\mathbb{Z}^{(q)}$ be the corresponding functions for $Y$.

All these scale functions appear in various fluctuation identities for $X$ and $Y$. For example, it is well known that, for $x \in (-\infty,b]$, we have
\begin{equation}\label{upward}
\esub{x}\left[e^{-q \tau_b^+} \I{\tau_b^+ < \tau_0^-} \right] = \frac{W^{(q)}(x)}{W^{(q)}(b)}
\end{equation}
and
\begin{equation}\label{upwardparisien}
\esub{x}\left[e^{-q \tau_b^+} \I{\tau_b^+ < \kappa_p}\right] = \frac{Z_{q,p}(x)}{Z_{q,p}(b)} ,
\end{equation}
as well as
\begin{equation}\label{downward}
\esub{x}\left[e^{-q \tau_0^-} \I{\tau_0^- < \tau_b^+}\right] = Z^{(q)}(x) - \frac{Z^{(q)}(b)}{W^{(q)}(b)}W^{(q)}(x)
\end{equation}
and
\begin{equation}\label{downwardinfini}
\esub{x}\left[e^{-q \tau_0^-} \I{\tau_0^- < \infty}\right] = Z^{(q)}(x) - \frac{q}{\Phi(q)} W^{(q)}(x).
\end{equation}
Clearly, we have equivalent identities for the stopping times related to $Y$ in terms of the corresponding scale functions.

\subsection{Convexity properties of scale functions}\label{section:convexity}

Some analytical properties of scale functions are going to play a fundamental role in the following sections.

It is known that, if the Lévy measure $\nu$ has a density, then $W^{(q)}$ is continuously differentiable on $(0, \infty)$; see, e.g., \cite{chan-kyprianou-savov_2011}. In Section~\ref{verificationsection}, we will need this density to be completely monotone.

\begin{lem}[\cite{loeffen_2009}]\label{scalemonotone}
Suppose that the L\'evy measure of $X$ has a completely monotone density. For $q > 0$, there exists a completely monotone function $f$ such that
\begin{equation}\label{difference}
W^{(q)}(x) = \frac{1}{\psi^\prime(\Phi(q))} e^{\Phi(q)x} - f(x), \quad x \geq 0.
\end{equation}
Moreover, $W^{(q) \prime}$ is a log-convex function on $(0, \infty)$.
\end{lem}

If the L\'evy measure $\nu$ has a completely monotone density, then $W^{(q) \prime}$ is infinitely differentiable and it is a strictly convex function on $(0, \infty)$ that is decreasing on $(0, a^*)$ and increasing on $(a^*, \infty)$, where
\begin{equation}\label{aetoile}
a^* := \sup\set{x \geq 0 : W^{(q) \prime}(x) \leq W^{(q) \prime}(y), \text{ for all } y \geq 0} .
\end{equation}

\begin{rem}
Recall from \cites{avram-palmowski-pistorius_2007,loeffen_2008} that $a^*$ is the optimal barrier level in the classical singular version of de Finetti's optimal dividends problem.
\end{rem}

As we assume that $p>0$, we can write
\begin{equation*}
Z_{q,p} (x) = p \int_0^\infty e^{-\Phi(p+q) y} W^{(q)}(x+y) \mathrm{d}y , \quad x \in \R .
\end{equation*}
Then, for $x > 0$, we have
\begin{equation}\label{Zqprimeutile}
Z_{q,p}^\prime (x) = p \int_0^\infty e^{-\Phi(p+q) y} W^{(q)\prime}(x+y) \mathrm{d}y .
\end{equation}
Clearly, $x \mapsto Z_{q,p} (x)$ is a nondecreasing continuous function.

The following result is lifted from \cite{renaud_2019} and will also be needed in Section~\ref{verificationsection}.

\begin{lem}[\cite{renaud_2019}]\label{lemma:log-convexity}
If $W^{(q) \prime}$ is log-convex on $(0,\infty)$, then $Z_{q,p}^\prime$ is log-convex on $(0,\infty)$.
\end{lem}

Combining Lemma~\ref{scalemonotone} and Lemma~\ref{lemma:log-convexity}, we deduce that, if the L\'evy measure $\nu$ has a completely monotone density, then $Z_{q,p}^\prime$ is infinitely differentiable and it is a log-convex function on $(0, \infty)$ that is decreasing on $(0, c^*)$ and increasing on $(c^*, \infty)$, where
\begin{equation}\label{cetoile}
c^* = \sup\set{x \geq 0 : Z_{q,p}'(x) \leq Z_{q,p}'(y), \text{ for all } y \geq 0} .
\end{equation}
Note that $c^* < \infty$ since $Z_{q,p}^\prime(\infty) = \infty$, the latter being inherited from $W^{(q) \prime}$ via~\eqref{Zqprimeutile}. For more details on these arguments, see \cite{renaud_2019}.

\begin{rem}
Note that these properties for $Z_{q,p}^\prime$ are also true under the mildest condition that the tail of the Lévy measure $\nu$ is log-convex. However, in Section~\ref{verificationsection}, we will also need the representation in~\eqref{difference} of Lemma~\ref{scalemonotone}.
\end{rem}

\section{Refraction strategies}\label{refractionsection}

In this section, as we are guessing that a refraction strategy should be optimal, we are first going to compute the performance function of an arbitrary refraction strategy. Let the performance function of $\pi_b$, that is the refraction strategy at level $b$, be given by
\begin{equation} \label{performance-refraction-definition}
V_b(x) = \esub{x} \left[\int_0^{\kappa_p^b} K e^{-qt} \I{U_t^b > b} \mathrm{d}t \right], \quad x \in \R.
\end{equation}

\begin{prop}\label{performance-proposition}
For $b \geq 0$, we have
\begin{equation}\label{performancefunction}
V_b(x) = \twopartdef{\frac{Z_{q,p}(x)}{h_p(b)}}{x \leq b,}{\frac{Z_{q,p}(x) + K\int_b^x \mathbb{W}^{(q)}(x-y) \left( Z_{q,p}'(y)-h_p(b) \right) \mathrm{d}y}{h_p(b)}}{x \geq b,}
\end{equation}
where
\begin{equation}\label{definitionh}
h_p(b) = \Phi_K(q) \int_0^\infty e^{-\Phi_K(q)y} Z_{q,p}'(b+y) \mathrm{d}y .
\end{equation}
\end{prop}

\begin{proof}
For $x \leq b$, we get
\begin{equation*}
V_b(x) = \esub{x}\left[e^{-q \tau_b^+} \I{\tau_b^+ < \kappa_p}\right] V_b(b) = \frac{Z_{q,p}(x)}{Z_{q,p}(b)} V_b(b),
\end{equation*}
where we used the strong Markov property, the fact that $X$ is a SNLP and the Parisian fluctuation identity given by~\eqref{upwardparisien}. 

Now, for $x \geq b$. We have, using the strong Markov property again, 
\begin{align*}
V_b(x) &= \esub{x}\left[\int_0^{\nu_b^-} Ke^{-qt} \mathrm{d}t \right] + \esub{x}\left[e^{-q \nu_b^-}  V_b\left(Y_{\nu_b^-}\right) \I{\nu_b^- < \infty}\right] \\
&= \frac{K}{q}\left(1 - \esub{x}\left[e^{-q \nu_b^-} \I{\nu_b^- < \infty}\right]\right) + \frac{\esub{x}\left[e^{-q \nu_b^-} Z_{q,p}\left(Y_{\nu_b^-}\right) \I{\nu_b^- < \infty} \right]}{Z_{q,p}(b)}V_b(b).
\end{align*}

Using the classical fluctuation identity given by~\eqref{downwardinfini} for $Y$ and the identity given by~\eqref{new-identity-equation} of Lemma~\ref{new-identity-lemma}, and noticing that
\begin{align*}
K \mathbb{W}^{(q)}(x-b) \int_0^\infty e^{-\Phi_K(q) z} Z_{q,p}^\prime(b+z) \mathrm{d}z = \frac{K}{\Phi_K(q)} \mathbb{W}^{(q)}(x-b) h_p(b),
\end{align*}
it results that
\begin{equation}\label{Vbpresque}
V_b(x) = \twopartdef{\frac{Z_{q,p}(x)}{Z_{q,p}(b)} V_b(b)}{x \leq b,}{\frac{K}{q}\left(1 - \mathbb{Z}^{(q)}(x-b) + \frac{q}{\Phi_K(q)}\mathbb{W}^{(q)}(x-b)\right) + \frac{G_b(x)V_b(b)}{Z_{q,p}(b)}}{x \geq b,}
\end{equation}
where
\begin{equation*}
G_b(x) := p \int_0^\infty e^{-\Phi(p+q)y} w_b^{(q)}(x;-y) \mathrm{d}y -\frac{K}{\Phi_K(q)} \mathbb{W}^{(q)}(x-b) h_p(b), \quad x \geq b \geq 0.
\end{equation*}
and where
\begin{equation*}
w_b^{(q)}(x;y) := W^{(q)}(x-y) + K\I{x \geq b} \int_b^x \mathbb{W}^{(q)}(x-z) W^{(q) \prime}(z-y) \mathrm{d}z, \quad x, y \in \R, b \geq 0.
\end{equation*}
Note that we have $w_{b+y}^{(q)}(b+y; 0) = w_b^{(q)}(x;-y)$, for all $x,y \in \R$ and for all $b \geq 0$.

Now, we must compute $V_b(b)$. First, let us assume $X$ has paths of bounded variation. We can use~\eqref{Vbpresque} to solve for $V_b(b)$ in~\eqref{Vbpresque}. Before doing so, note that we have $\mathbb{Z}^{(q)}(0) = 1$ and
\begin{equation*}
G_b(b) = Z_{q,p}(b) - \frac{K}{\Phi_K(q)} \mathbb{W}^{(q)}(0) h_p(b).
\end{equation*}
Consequently, 
\begin{align*}
V_b(b) = \left(\frac{K}{\Phi_K(q)} \mathbb{W}^{(q)}(0)\right) \left(1 - \frac{G_b(b)}{Z_{q,p}(b)}\right)^{-1} = \frac{Z_{q,p}(b)}{\Phi_K(q)\int_0^\infty e^{-\Phi_K(q) z} Z_{q,p}^\prime(b+z) \mathrm{d}z} = \frac{Z_{q,p}(b)}{h_p(b)}.
\end{align*}

Now, assume $X$ has paths of UBV. In this case, we get the same expression for $V_b(b)$ by using the following approximation argument: for $n$ large enough, let $X^n$ be the bounded variation SNLP with L\'evy triplet $(\gamma, 0, \nu^n)$, in which
\begin{equation*}
\nu^n(\dd x) = \I{x \geq 1/n}\nu(\dd \theta) + \sigma^2 n^2 \delta_{1/n}(\dd x),
\end{equation*}
where $\delta_{1/n}(\dd x)$ stands for the Dirac point mass measure at $1/n$. It is well known (see \cite{bertoin_1996}) that $X^n$ converges to $X$ on compact time intervals, i.e., for all $t > 0$, we have
\begin{equation*}
\lim_{n \to \infty} \sup_{0 \leq s \leq t} \left|X_s^n - X_s\right| = 0.
\end{equation*}
Hence, defining $\kappa_p^{n,b}$ as the termination time when using the surplus process $X^n$ with the refraction strategy at level $b$, and defining similarly $L_t^{b,n}$ as the corresponding dividend process, it can be shown that we have, $\mathbb{P}_b$-a.s., \begin{equation*}
\lim_{n \to \infty}\int_0^{\kappa_p^{n,b}} e^{-qt} \dd L_t^{n,b} = \int_0^{\kappa_p^b} e^{-qt} \dd L_t^b.
\end{equation*}
Since we have
\begin{equation*}
\int_0^{\kappa_p^{n,b}} e^{-qt} \dd L_t^{n,b} \leq \int_0^\infty Ke^{-qt} \dd t = \frac{K}{q} ,
\end{equation*}
it follows by the dominated convergence theorem that
\begin{equation}\label{1sthalf}
\lim_{n \to \infty} \esub{b}\left[\int_0^{\kappa_p^{n,b}} e^{-qt} \dd L_t^{n,b}\right] = \esub{b}\left[\int_0^{\kappa_p} e^{-qt} \dd L_t^b\right] .
\end{equation}

Next, it is also well known that, as $n$ tends to infinity, the Laplace exponent of $X^n$ converges to the Laplace exponent of $X$, which means, by the continuity theorem for Laplace transforms, that the corresponding $q$-scale functions also converge. In other words, defining $Z_{q,p}^n$ and $h_p^n$ as the corresponding functions for the surplus process $X^n$, it follows that
\begin{equation}\label{2ndhalf}
\lim_{n \to \infty} \frac{Z_{q,p}^n(b)}{h_p^n(b)} = \frac{Z_{q,p}(b)}{h_p(b)}.
\end{equation}

Combining \eqref{1sthalf} and \eqref{2ndhalf}, and since, by the previous step in this proof (when $X$ has paths of BV), we have
\[
\esub{b}\left[\int_0^{\kappa_p^{n,b}} e^{-qt} \dd L_t^{n,b}\right]=\frac{Z_{q,p}^n(b)}{h_p^n(b)} ,
\]
it follows that when $X$ has paths of UBV, we also have \begin{equation*}
V_b(b) = \frac{Z_{q,p}(b)}{h_p(b)}.
\end{equation*}

In conclusion, for $x \geq b$, we have
\begin{align*}
V_b(x) &= \frac{K}{q}\left(1 - \mathbb{Z}^{(q)}(x-b)\right) + \frac{p \int_0^\infty e^{-\Phi(p+q)y} w_b^{(q)}(x;-y) \mathrm{d}y}{h_p(b)} \\
&= -K\int_0^{x-b} \mathbb{W}^{(q)}(y) \mathrm{d}y + \frac{Z_{q,p}(x) + K\int_b^x \mathbb{W}^{(q)}(x-y) Z_{q,p}^\prime(y) \mathrm{d}y}{h_p(b)}.
\end{align*}
Putting all the pieces together, the result follows.
\end{proof}

Note that, for fixed $q,K > 0$, the map $p \mapsto \Phi_K(q) - \Phi(p+q)$ is strictly decreasing and goes to $-\infty$ as $p$ tends to $\infty$, thanks to the fact that $\Phi$ is strictly increasing and $\Phi(\infty) = \infty$. Since $\Phi_K(q) > \Phi(q)$, there exists a unique $p_0 > 0$ such that
\begin{equation}\label{p_0}
\Phi_K(q) - \Phi(p_0 + q) = 0 .
\end{equation}
It is easy to verify that $\Phi(p_0 + q) = p_0 / K$.

Consequently, let us define the following quantity: for $K,p > 0$,
\begin{equation}
\label{C(K,p)extended}
C(K,p) = \twopartdef{\frac{\Phi(p+q) - p/K}{\Phi_K(q) - \Phi(p+q)}}{p \neq p_0,}{\Phi_K'(p_0 / K)}{p = p_0.}
\end{equation}

Now, note that
\begin{equation}\label{h(0)partiel}
h_p(0) = \Phi_K(q) C(K,p),
\end{equation}
which follows readily from the following Laplace transform: for $\theta>\Phi(q)$,
\[
\int_0^\infty e^{-\theta y} Z_{q,p}(y) \mathrm{d}y = \left( \frac{1}{\psi(\theta)-q} \right) \frac{\psi(\theta)-(p+q)}{\theta-\Phi(p+q)} .
\]
The details are left to the reader.

In conclusion, we can write
\[
V_0(x) =
\begin{cases}
(\Phi_K(q) C(K,p))^{-1}e^{\Phi(p+q) x} & \text{if $x \leq 0$,}\\
-K \int_0^{x} \mathbb{W}^{(q)}(y) \mathrm{d}y & \\
\qquad + (\Phi_K(q) C(K,p))^{-1} \left( Z_{q,p}(x) + K\int_0^x \mathbb{W}^{(q)}(x-y)Z_{q,p}'(y) \mathrm{d}y \right) & \text{if $x \geq 0$.}
\end{cases}
\]

\subsection{The optimal threshold}

We now want to find an optimal threshold $b^* \in [0,\infty)$ for which the refraction strategy $\pi_{b^*}$ outperforms all other refraction strategies. In view of Proposition~\ref{performance-proposition}, a natural candidate would be
\begin{equation}\label{eq:def-of-b-star}
b^* := \mathrm{argmin}_{y \geq 0} h_p(y).
\end{equation}
Under our standing assumption, this quantity is well defined, thanks to the next lemma in which the required analytical properties of $h_p$ are obtained.

\begin{lem}\label{caracteristiques}
If the L\'evy measure of $X$ has a completely monotone density, then $h_p$ is a nonnegative, infinitely differentiable and strictly convex function on $(0, \infty)$ such that $h_p(\infty) = \infty$. In particular, a minimizer of $h_p$ exists and is unique.
\end{lem}

\begin{proof}
Under our assumption, $h_p$ is nonnegative and infinitely differentiable because $Z_{q,p}'$ is nonnegative and infinitely differentiable. Next, recall that combining Lemma~\ref{scalemonotone} and Lemma~\ref{lemma:log-convexity}, we get that $Z_{q,p}^\prime$ is log-convex on $(0, \infty)$. As a consequence, we can prove that $h_p$ is also log-convex on $(0, \infty)$, using the properties of log-convex functions, as in \cite{renaud_2019} for $Z_{q,p}^\prime$. In particular, we have that $h_p$ is strictly convex on $(0, \infty)$. 

Next, we prove that $h_p(\infty) = \infty$. Using integration by parts for the first inequality, we can write
\begin{align*}
h_p(b) &\geq (\Phi_K(q))^2 \int_0^\infty e^{- \Phi_K(q)u}\left(Z_{q,p}(u+b) - Z_{q,p}(b)\right) \mathrm{d}u \\
&= p (\Phi_K(q))^2 \int_0^\infty e^{- \Phi_K(q) u} \left\lbrace \int_0^\infty e^{-\Phi(p+q)y} \left[ W^{(q)}(u+b+y) - W^{(q)}(b+y) \right] \mathrm{d}y \right\rbrace \mathrm{d}u \\
&=  p (\Phi_K(q))^2 \int_0^\infty e^{- \Phi_K(q) u} \\
&\qquad \times \left\{ \int_0^\infty e^{-\Phi(p+q)y} e^{\Phi(q)(b+y)} \left[e^{\Phi(q)u}W_{\Phi(q)}(u+b+y) - W_{\Phi(q)}(b+y)\right] \mathrm{d}y \right\}\mathrm{d}u \\
&\geq p (\Phi_K(q))^2 \int_0^\infty e^{- \Phi_K(q) u} \left\{\int_0^\infty e^{-\Phi(p+q)y} e^{\Phi(q)(b+y)} W_{\Phi(q)}(b+y)\left[e^{\Phi(q)u} - 1\right] \mathrm{d}y \right\}\mathrm{d}u \\
&= (\Phi_K(q))^2 \left[ \int_0^\infty e^{-\Phi_K(q)u} \left(e^{\Phi(q)u} - 1\right) \mathrm{d}u\right] Z_{q,p}(b) \\
&= \frac{\Phi_K(q) \Phi(q)}{\Phi_K(q) - \Phi(q)} Z_{q,p}(b),
\end{align*}
where $W_{\Phi(q)}$ is the $0$-scale function for the SNLP $X$ with respect to the probability measure $\mathbf{P}_x^{\Phi(q)}$ given by the Esscher transformation in~\eqref{esscher}. Since $\Phi_K(q) > \Phi(q)$ and $Z_{q,p}(\infty) = \infty$, we can conclude that $h_p(\infty) = \infty$. 
\end{proof}

In the next proposition, we give some important properties of this (candidate) optimal threshold.

\begin{prop}\label{specifications-sur-b^*}
Assume the L\'evy measure of $X$ has a completely monotone density. We have $0 \leq b^* \leq c^*$. Also, $b^* > 0$ if and only if
\begin{equation}\label{condition-b^*}
\Phi_K(q) C(K,p) <  \Phi(p+q) - p W^{(q)}(0+).
\end{equation}
\end{prop}

\begin{proof}
First, note that $h_p(b)$ can be written as the expectation of the random variable $Z_{q,p}^\prime(\mathbf{e}+b)$, where $\mathbf{e}$ is an (independent) exponential random variable with mean $1/\Phi_K(q)$. Since $Z_{q,p}^\prime$ is increasing on $(c^*, \infty)$ under our assumption, then so is $h_p$. This proves that $b^* \leq c^*$.

Next, let us prove that $b^* > 0$ if and only if
\begin{equation}\label{condition-generale}
h_p(0) < Z_{q,p}'(0+).
\end{equation}
As
\begin{equation*}
h_p(b) = \Phi_K(q) \int_0^\infty e^{-\Phi_K(q) y} Z_{q,p}'(b+y) \mathrm{d}y = \Phi_K(q) e^{\Phi_K(q) b} \int_b^\infty e^{- \Phi_K(q) z} Z_{q,p}'(y) \mathrm{d}y ,
\end{equation*}
then we have
\begin{equation}\label{liencondition}
h_p'(b) = \Phi_K(q)\left(h_p(b) - Z_{q,p}'(b)\right) .
\end{equation}
By the strict convexity of $h_p$, it is clear that $b^* > 0$ if and only if $h_p'(0) < 0$. We deduce from~\eqref{liencondition} that $b^* > 0$ if and only if~\eqref{condition-generale} is verified. Finally, recalling \eqref{h(0)partiel} and  using the definition of $Z_{q,p}$ in~\eqref{Zphi(p+q)}, we deduce that $Z_{q,p}^\prime(0+)=\Phi(p+q)-p W^{(q)}(0+)$. The result follows.

\end{proof}

\begin{rem}
In inequality~\eqref{condition-b^*}, the sign of $\Phi_K(q) - \Phi(p+q)$ plays an important role. Indeed, for example if $p>p_0$, or equivalently if $\Phi_K(q) < \Phi(p+q)$, then using~\eqref{C(K,p)extended} and elementary algebraic manipulations in \eqref{condition-b^*}, we deduce that $b^* > 0$ if and only if
\begin{equation}\label{condition-interesting}
\Phi(p+q)\left(\frac{\Phi(p+q)}{p} - W^{(q)}(0)\right)
> \Phi_K(q)\left(\frac{1}{K} - W^{(q)}(0)\right).
\end{equation}
\end{rem}

In the last proof, we have obtained that, if $b^*>0$, then it is the unique root/solution of
\begin{equation}\label{equationunique}
h_p(b) = Z_{q,p}'(b), \quad b > 0.
\end{equation}
From this, we deduce the following corollary to Proposition~\ref{performance-proposition}:
\begin{cor}\label{cor:vbetoile}
If $b^*>0$, then
\[
V_{b^*}(x) = \twopartdef{\frac{Z_{q,p}(x)}{Z_{q,p}^\prime(b^*)}}{x \leq b^*,}{\frac{Z_{q,p}(x) + K \int_{b^*}^x \mathbb{W}^{(q)}(x-y) \left( Z_{q,p}'(y)-Z_{q,p}^\prime(b^*) \right) \mathrm{d}y}{Z_{q,p}^\prime(b^*)}}{x \geq b^*.}
\]
\end{cor}

It is interesting to note that $V_{b^*}$ is continuously differentiable at $x=b^*$ and that $V_{b^*}^\prime(b^*)=1$.

\subsection{Relationships with the limiting control problems}\label{sect:relationships}

We conclude this section by showing that the previous results generalize those obtained for the two limiting control problems considered in \cite{kyprianou-et-al_2012}) and in \cite{renaud_2019}. On one hand, if $p \to \infty$, then the delays become negligible, which means we are back to classical ruin. Mathematically, when $p \to \infty$, we have
\begin{equation*}
\frac{\left(\Phi(p+q)\right)^2}{p} - \Phi(p+q)W^{(q)}(0) = \frac{\Phi(p+q)}{p}Z_{q,p}^\prime(0+) \longrightarrow W^{(q) \prime}(0+),
\end{equation*}
which means that~\eqref{condition-interesting} becomes
\[
W^{(q) \prime}(0+) > \Phi_K(q) \left(\frac{1}{K} - W^{(q)}(0)\right).
\]
The latter coincides with Lemma 3 in \cite{kyprianou-et-al_2012}.

On the other hand, if $K \to \infty$ (and the underlying process is not of BV), then the dividend rates are allowed to take very large values as in the singular control problem. Mathematically, when $K \to \infty$, note that $h_p(0)$ increases (resp.\ decreases) to $Z_{p,q}^\prime(0+)$ if and only if $Z_{q,p}''(0+) < 0$ (resp.\ $Z_{q,p}''(0+) > 0$). In that case, \eqref{condition-generale} becomes
\[
Z_{q,p}''(0+) < 0 .
\]
The latter coincides with Proposition 2 in \cite{renaud_2019}.

\section{Verification of optimality}\label{verificationsection}

In Section~\ref{refractionsection}, we identified a candidate optimal strategy for our control problem, namely the refraction strategy at optimal level $b^*$. As described in the introduction, we will now prove that $\pi_{b^*}$ is indeed an optimal admissible strategy using a verification lemma.

First, let us define a class of functions:
\begin{dfn}
Let $\cS$ to be the class of functions $g \colon \R \to \R$ such that:
\begin{enumerate}
\item $g$ is sufficiently smooth, i.e., $g \in C^1(\R) \cap C^2(\R \setminus \{0\})$ when $X$ has paths of UBV, and $g \in C^0(\R) \cap C^1(\R \setminus \{0\})$ when $X$ has paths of BV;
\item $g'(0-) \geq g'(0+)$;
\item $g$ can be written as the difference of two convex functions.
\end{enumerate}
\end{dfn} 

In what follows, for a function $g \in \cS$,  we define $g^\prime(0) = g^\prime(0-)$ when $X$ has paths of BV, and $g^{\prime \prime}(0) = g^{\prime \prime}(0-)$ when $X$ has paths of UBV. As a consequence, $g^\prime \colon \R \to \R$ (resp.\  $g^{\prime \prime} \colon \R \to \R$) is now well defined. Finally, let $\Gamma$ be the differential operator acting on functions in $\cS$ and given by
\begin{equation}\label{def:gamma}
\Gamma g(x) = \gamma g^\prime(x) + \frac{\sigma^2}{2} g^{\prime \prime}(x) + \int_{(0, \infty)} \left(g(x-z) - g(x) + g^\prime(x) z \mathbf{1}_{(0, 1]} \right) \nu(\mathrm{d}z), \quad x \in \R .
\end{equation}


\begin{lem}
If the L\'evy measure of $X$ has a completely monotone density, then $V_{b^*} \in \cS$.
\end{lem}

\begin{proof}
First, assume that $b^* > 0$ and let us prove that $V_{b^*}$ is sufficiently smooth.

From~\eqref{Zphi(p+q)}, we can write
\begin{equation*}
Z_{p,q}'(x) =
\begin{cases}
\Phi(p+q)Z_{p,q}(x) & \text{if $x<0$,}\\
\Phi(p+q)Z_{p,q}(x) - p W^{(q)}(x) & \text{if $x>0$,}
\end{cases}
\end{equation*}
and
\begin{equation*}
Z_{p,q}''(x) = \twopartdef{\Phi^2(p+q)Z_{p,q}(x)}{x < 0,}{\Phi^2(p+q)Z_{p,q}(x) - p\Phi(p+q)W^{(q)}(x) - pW^{(q) \prime}(x)}{x > 0.}
\end{equation*}
Recall that, since the Lévy measure has a density, it is known that $W^{(q)}$ is continuously differentiable on $(0,\infty)$. Clearly, we have that $Z_{p,q}$ is sufficiently smooth.

From the above, we deduce that $Z_{p,q}'(0-) = \Phi(p+q)$, $Z_{p,q}'(0+) = \Phi(p+q) - p W^{(q)}(0+)$ and $Z_{p,q}''(0-) = \Phi^2(p+q)$. Using~ \eqref{Wen0}, we can conclude that $Z_{p,q} \in \cS$.

Then, from Corollary~\ref{cor:vbetoile}, we get that $V_{b^*}$ is sufficiently smooth. Also, it is clear that $V_{b^*}'(0-) \geq V_{b^*}'(0+)$.

Now, assume that $b^* = 0$. It is clear from the previous analysis that $V_0$ is sufficiently smooth. Also, we have
\begin{equation*}
V_0'(0-) = \frac{Z_{p,q}'(0-)}{h_p(0)}
\end{equation*}
and
\begin{equation*}
V_0'(0+) = \frac{Z_{p,q}'(0+)}{h_p(0)} + \frac{K \bbW^{(q)}(0)}{h_p(0)}\left(Z_{p,q}'(0+) - h_p(0)\right).
\end{equation*}  
Recalling that $b^* = 0$ if and only if $h_p(0) \geq Z_{q,p}'(0+)$ (see~\eqref{condition-generale}), it follows that $V_{0}'(0-) \geq V_{0}'(0+)$.
Finally, following the same line of reasoning as above, we also have that $V_{0}$ can be written as the difference of two convex functions, which concludes the proof.
\end{proof}

Here is the verification lemma of our control problem.

\begin{lem}\label{verificationlemma}
For $\hat{\pi} \in \Pi_K$, if its performance function $V_{\hat{\pi}}$ belongs to $\cS$ and is such that, for all $x \in \R$, 
\begin{equation}\label{verlemmaeq}
\sup_{0 \leq u \leq K} \left[ \left( \Gamma - q - p \mathbf{1}_{(-\infty, 0)}(x) \right) V_{\hat{\pi}}(x) + u(1 - V_{\hat{\pi}}^\prime(x)) \mathbf{1}_{[0, \infty)}(x)\right] \leq 0 ,
\end{equation}
where $\Gamma$ is defined in~\eqref{def:gamma}, then $\hat{\pi}$ is an optimal strategy.
\end{lem} 

\begin{proof}
Let $\pi$ be an arbitrary admissible strategy. For simplicity, in what follows, we set $g := V_{\hat{\pi}}$, $U := U^\pi$ and $l := l^\pi$. Since $U$ is a semimartingale and since $g \in \cS$, we can apply the Meyer-Ito formula (see Theorem 70 in \cite{protter_2005}) to obtain
\begin{multline*}
g(U_t) = g(U_0) + \int_{0+}^t g'(U_{s-}) \mathrm{d}U_s + \frac{1}{2} \int_{-\infty}^\infty L_t^y \mu(\dd y) \\
+ \sum_{0 < s \leq t} \left(g(U_s) - g(U_{s-}) - g'(U_{s-}) \Delta U_s \right),
\end{multline*}
where $\mu$ is the weak second derivative of $g$ and $L^y$ is the semimartingale local time at $y$ of $U$. In particular, we have $\mu(\dd y)=g''(y) \dd y + (g'(0+)-g'(0-)) \delta_0 (\dd y)$, in which the value of $g''(0)$ is not needed. If $\sigma>0$, then by the occupation time formula we have
\[
\int_{-\infty}^\infty L_t^y \mu(\dd y) = \sigma^2 \int_0^t g''(U_s) \dd s ,
\]
and, if $X$ has paths of BV, then
\[
\int_{-\infty}^\infty L_t^y \mu(\dd y) = (g'(0+)-g'(0-)) L_t^0 .
\]
Otherwise, this integral is equal to zero. Merging all those cases together, we can write
\[
\int_{-\infty}^\infty L_t^y \mu(\dd y) = \sigma^2 \int_0^t g''(U_s) \dd s - p W^{(q)}(0+) L_t^0 .
\]

Now, defining
\begin{equation*}
Y_t = e^{-qt - p \int_0^t \I{U_s<0} \dd s}
\end{equation*}
and performing a stochastic integration by parts, we can write
\begin{multline*}
Y_t g(U_t) = g(U_0) + \int_0^t Y_s (\Gamma - q - p\I{U_s < 0})g(U_s) \dd s + \int_0^t Y_s \left( 1 - g'(U_s) \right) l_s \dd s \\
- \int_0^t Y_s l_s \dd s - \frac{pW(0+)}{2} \int_{0+}^t Y_s \dd L^0_s + \int_{0+}^t Y_s \dd M_s ,
\end{multline*}
where
\begin{align*}
M_t &:= \int_{0+}^t g'(U_{s-}) \dd \left[X_s - \gamma s - \sum_{0 < u \leq s} \Delta X_u \I{|\Delta X_u| \geq 1} \right] \\
& \quad + \sum_{0 < s \leq t} \left[g(U_{s-} + \Delta X_s) - g(U_{s-}) - g'(U_{s-}) \Delta X_s \I{|\Delta X_s| < 1}\right] \\
& \qquad - \int_{0}^t \int_{0+}^\infty \left[g(U_{s-} - y) - g(U_{s-}) + g'(U_{s-}) y \I{0 < y < 1}\right] \nu(\dd y) \dd s
\end{align*}
is a (local) martingale. Since $g$ satisfies~\eqref{verlemmaeq}, we have
\begin{equation}\label{local}
0 \leq Y_t g(U_t) \leq g(U_0) - \int_0^t Y_s l_s \dd s + \tilde{M}_t ,
\end{equation}
where $\tilde{M}_t := \int_{0+}^t Y_s \dd M_s$ is a (local) martingale. 

As in \cite{renaud_2019}, we have that, for all $x \in \R$,
\begin{equation*}
\esub{x}\left[\int_0^t Y_s l_s \dd s \right] = \esub{x}\left[\int_0^{t \wedge \kappa_p} e^{-qs} l_s \dd s \right] .
\end{equation*}
Choosing a localizing sequence $\set{T_n : n \geq 0}$ for $\tilde{M}$ and taking expectations (with $x \in \R$) on each side of~\eqref{local}, we have
\begin{equation}\label{loeffenargument}
0 \leq g(x) - \esub{x}\left[\int_{0}^{t \wedge T_n \wedge \kappa_p} e^{-qs} l_s \dd s \right] .
\end{equation}
Taking the limits when $t,n \to \infty$, we get that $g(x) \geq V_\pi(x)$ for all $x \in \R$. Since $\pi \in \Pi_K$ is arbitrary, the result follows.
\end{proof}

The last step consists in proving that $V_{b^*}$, the performance function of the refraction strategy at level $b^*$, satisfies the (sufficient) conditions given in the verification lemma. 

\begin{lem}\label{analyseV}
If the L\'evy measure of $X$ has a completely monotone density, then
\begin{equation}\label{troisequations}
\threepartdef{(\Gamma - q - p)V_{b^*}(x) = 0,}{x < 0,}{(\Gamma - q)V_{b^*}(x) = 0,}{0 \leq x \leq b^*,}{(\Gamma - q)V_{b^*}(x) + K\left(1 - V_{b^*}^\prime(x)\right) = 0,}{x > b^*.}
\end{equation}
\end{lem}

\begin{proof}
The first two equations of~\eqref{troisequations} have already been verified in \cite{renaud_2019}. To prove the third equality, it suffices to follow the steps of the proof of Lemma 6 in \cite{kyprianou-et-al_2012} in which the time of classical ruin is replaced by the Parisian ruin time.
\end{proof}

As the final step of our verification procedure, we show that $V_{b^*}$ is a concave function on $[0,\infty)$, which is stronger than what is needed.  This is where the assumption of complete monotonicity of the density of the L\'evy measure is of paramount importance. Indeed, we will use a result of Lemma~\ref{scalemonotone}, saying that the $q$-scale function can be written as the difference of an exponential function and a completely monotone function, together with Bernstein's theorem.

\begin{prop}\label{hardest}
If the L\'evy measure of $X$ has a completely monotone density, then $V_{b^*}$ is a concave function on $[0,\infty)$.
\end{prop}

A proof of this last proposition is given in Appendix~\ref{proof-of-hardest-prop}.

In conclusion, here is the full solution to our stochastic control problem.

\begin{thm}\label{mainresult2}
Fix constants $p, q, K > 0$. If the L\'evy measure of $X$ has a completely monotone density, then:
\begin{enumerate}
\item if condition~\eqref{condition-b^*} is verified, then the refraction strategy at level $b^*>0$ given by~\eqref{equationunique} is optimal and
\[
V(x) = \twopartdef{\frac{Z_{q,p}(x)}{Z_{q,p}^\prime(b^*)}}{x \leq b^*,}{\frac{Z_{q,p}(x) + K \int_{b^*}^x \mathbb{W}^{(q)}(x-y) \left( Z_{q,p}'(y)-Z_{q,p}^\prime(b^*) \right) \mathrm{d}y}{Z_{q,p}^\prime(b^*)}}{x \geq b^*;}
\]

\item otherwise, if condition~\eqref{condition-b^*} is not verified, then the refraction strategy at level $b^*=0$ is optimal and
\[
V(x) =
\begin{cases}
\frac{\Phi_K(q) - \Phi(p+q)}{\Phi_K(q) \left(\Phi(p+q) - p/K \right)} e^{\Phi(p+q) x} & \text{if $x \leq 0$,}\\
-K \int_0^{x} \mathbb{W}^{(q)}(y) \mathrm{d}y & \\
\qquad + \frac{\Phi_K(q) - \Phi(p+q)}{\Phi_K(q) \left(\Phi(p+q) - p/K \right)} \left( Z_{q,p}(x) + K\int_0^x \mathbb{W}^{(q)}(x-y)Z_{q,p}'(y) \mathrm{d}y \right) & \text{if $x \geq 0$.}
\end{cases}
\]
\end{enumerate}
\end{thm}

\section{Sensitivity analysis of the optimal refraction level}\label{sensitivity}

In this section, we analyze the sensitivity of the optimal dividend strategy with respect to some of the parameters, especially the Parisian delay $p$ and the payment rate bound $K$. 

Recall from Section~\ref{sect:relationships} that, if $p$ is large, then the delays are small, which means the Parisian ruin time will likely occur not long after the classical ruin time. Consequently, in a model for which our solution applies, it is expected then that $b^*$ will be close to $b^*_\infty$, the optimal refraction level obtained in \cite{kyprianou-et-al_2012}.

Conversely, if $p$ is small, then delays are large, which means that (Parisian) ruin is less likely. Thus, it is expected that $b^*$ will be close to zero as the risk of (Parisian) ruin is very low.

Finally, recall also from Section~\ref{sect:relationships} that, if $K$ is large, then the payment rates are allowed to take large values which will make the optimal refraction strategy with level $b^*$ close to the optimal strategy of the corresponding singular control problem studied in \cite{renaud_2019}, which is the barrier strategy at level $c^*$. Recall that we have shown in Proposition~\ref{specifications-sur-b^*} that $b^* \leq c^*$.

Let us illustrate all of this in the following two models (each satisfying our standing assumption on the Lévy measure): a Brownian risk model and a Cram\'er-Lundberg model with exponentially distributed claims.

\subsection{Brownian risk model}

Let $X$ be a Brownian motion with drift, i.e.,
\begin{equation*}
X_t = X_0 + \mu t + \sigma B_t, \quad t \geq 0 ,
\end{equation*}
where $B = \set{B_t : t \geq 0}$ is a standard Brownian motion, and where $\mu \in \R$ and $\sigma > 0$.

Recall that the optimal refraction level $b^*$ is the minimum of $h_p(b)$, which is defined in \eqref{definitionh}. Straightforward computations lead to \begin{equation*}
h_p(b) = D_1 e^{\lambda_1 b} + D_2 e^{-\lambda_2 b}, \quad b \geq 0,
\end{equation*}
where
\begin{equation*}
\lambda_1 = \frac{\Delta - \mu}{\sigma^2}, \qquad \lambda_2 = \frac{\Delta + \mu}{\sigma^2},
\end{equation*}
where
\begin{equation*}
D_1 = \frac{p \lambda_1 \Phi_K(q)}{\Delta (\Phi(p+q) - \lambda_1) (\Phi_K(q) - \lambda_1)}, \qquad D_2 = \frac{p \lambda_2 \Phi_K(q)}{\Delta (\Phi(p+q) + \lambda_2) (\Phi_K(q) + \lambda_2)},
\end{equation*}
and where $\Delta = \sqrt{\mu^2 + 2\sigma^2 q}$. Note that $\Phi(q) = \lambda_1$.

First, in Figure~\ref{fig:pBrownian}, we see that $p \mapsto b^*$ increases from zero to $b^*_\infty$ when $p$ increases, as discussed above.

\begin{figure}[h!]
\centering
\textbf{Optimal level as a function of the Parisian rate in a Brownian risk model}\par
\includegraphics{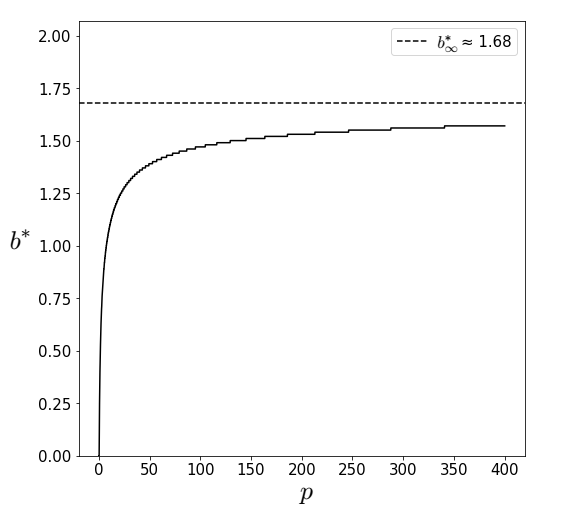}
\caption{Model and problem parameters: $\mu = 11, \sigma = 3, q = 1, K = 10$.}
\label{fig:pBrownian}
\end{figure}

Second, in Figure~\ref{fig:kBrownian}, we see that $K \mapsto b^*$ increases toward $c^*$ when $K$ increases, as discussed above.

\begin{figure}[h!]
\centering
\textbf{Optimal level as a function of the control rate bound in a Brownian risk model}\par
\includegraphics{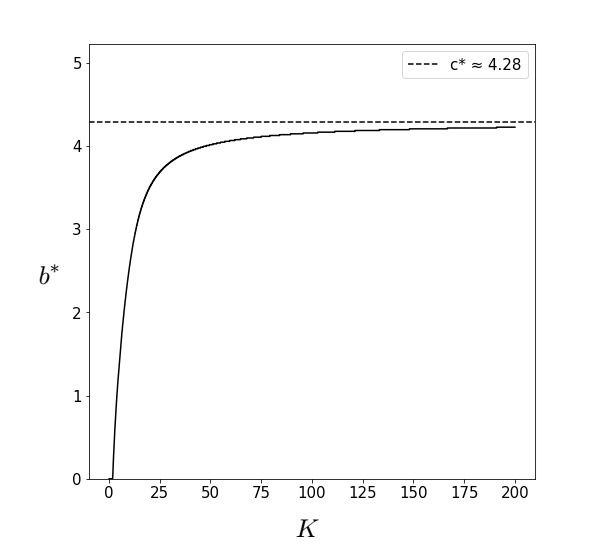}
\caption{Model and problem parameters: $\mu = 9, \sigma = 5, q = 1, p = 4$.}
\label{fig:kBrownian}
\end{figure}

Intuitively, as for the model with classical ruin, if the volatility parameter $\sigma$ is small, then the drift \textit{dominates} and it is likely the process will get away from the \textit{red zone} and avoid starting an exponential clock, so we expect $b^\ast$ to be small. Similarly, if the volatility parameter $\sigma$ is large, then the volatility \textit{dominates} and it is more likely that the process will go below zero, so we expect $b^\ast$ to be small since the process is not viable and liquidation should be declared as soon as possible. This is illustrated in Figure~\ref{fig:volatility}.

\begin{figure}[h!]
\centering
\textbf{Optimal level as a function of the volatility in a Brownian risk model}\par
\includegraphics{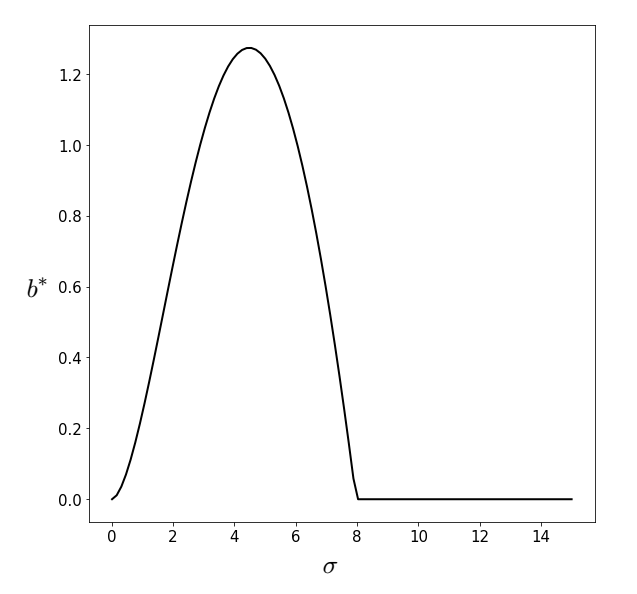}
\caption{Model and problem parameters: $\mu = 10, q = 1, p = 5, K = 5$.}
\label{fig:volatility}
\end{figure}

\subsection{Cram\'er-Lundberg model with exponentially distributed claims}

Let $X$ be a compound Poisson process with drift and exponentially distributed claims, i.e.,
\begin{equation*}
X_t = X_0 + ct - \sum_{i=1}^{N_t} C_i, \quad t \geq 0,
\end{equation*}
where $c > 0$, $N = \set{N_t : t \geq 0}$ is a Poisson process with intensity $\lambda > 0$, and where the $C_i$'s are independent and all exponentially distributed with mean $1/\alpha$. Consequently, we obtain
\begin{equation*}
h_p(b) = E_1 e^{\theta_1 b} + E_2 e^{\theta_2 b},
\end{equation*} 
where $\theta_2 < 0 < \theta_1$ are the two distinct roots of
\begin{equation*}
c\theta - \frac{\lambda \theta}{\alpha + \theta} - q = 0, \quad \theta \in \R,
\end{equation*}
where
\begin{equation*}
E_1 = \frac{p F_1 \theta_1 \Phi_K(q)}{(\Phi(p+q) - \theta_1)(\Phi_K(q) - \theta_1)}, \qquad E_2 = \frac{p F_2 \theta_2 \Phi_K(q)}{(\Phi(p+q) - \theta_2)(\Phi_K(q) - \theta_2)},
\end{equation*}
and where
\begin{equation*}
F_1 = \frac{(\alpha + \theta_1)^2}{c(\alpha + \theta_1)^2 - \lambda \alpha}, \qquad F_2 = \frac{(\alpha + \theta_2)^2}{c(\alpha + \theta_2)^2 - \lambda \alpha}.
\end{equation*}
Note that $\Phi(q) = \theta_1$.

First, in Figure~\ref{fig:pPoisson}, we see that $p \mapsto b^*$ also increases from zero to $b^*_\infty$ when $p$ increases, as discussed above and as in the Brownian model considered above.

\begin{figure}[h!]
\centering
\textbf{Optimal level as a function of the Parisian rate in a Cram\'er-Lundberg model}\par
\includegraphics{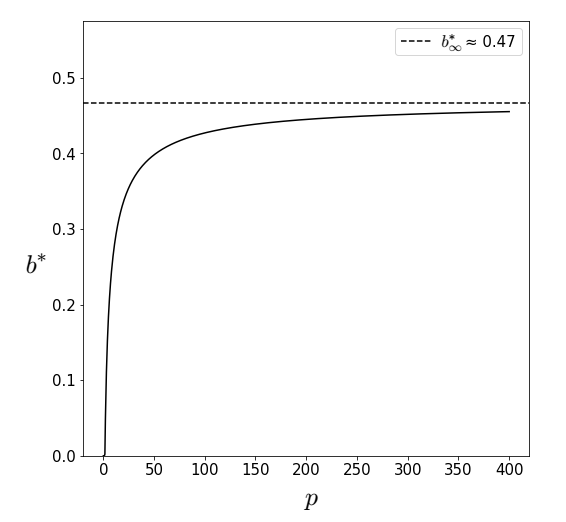}
\caption{Model and problem parameters: $c = 5, \lambda = 20, \alpha = 8, q = 2, K = 4$.}
\label{fig:pPoisson}
\end{figure}

Second, in Figure~\ref{fig:kPoisson}, we see that $K \mapsto b^*$ increases toward $c^*$ when $K$ increases up to $c$, which is the maximal value allowed for $K$ since, in this model, the surplus process has paths of BV.

\begin{figure}[h!]
\centering
\textbf{Optimal level as a function of the control rate bound in a Cram\'er-Lundberg model}\par
\includegraphics{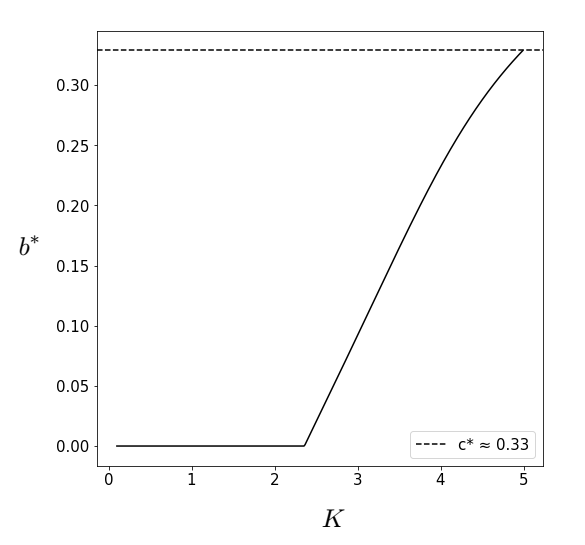}
\caption{Model and problem parameters: $c = 5, \lambda = 12, \alpha = 8, q = 2, p = 10$.}
\label{fig:kPoisson}
\end{figure}

Finally, let us look at the sensitivity of the optimal refraction level with respect to the claim frequency parameter, which will prove to be similar to the one for the volatility parameter in a Brownian risk model. Indeed, if the claim frequency parameter $\lambda$ is small, then the drift \textit{dominates} and it is likely the process will get away from the \textit{red zone} and avoid starting an exponential clock, so we expect $b^\ast$ to be small. Similarly, if the claim frequency $\lambda$ is large, then the jumps \textit{dominates} and it is more likely that the process will go below zero, so we expect $b^\ast$ to be small since the process is not viable and liquidation should be declared as soon as possible. This is illustrated in Figure~\ref{fig:rateofclaim}.

\begin{figure}[h!]
\centering
\textbf{Optimal level as a function of the claim frequency in a Cram\'er-Lundberg model}\par
\includegraphics{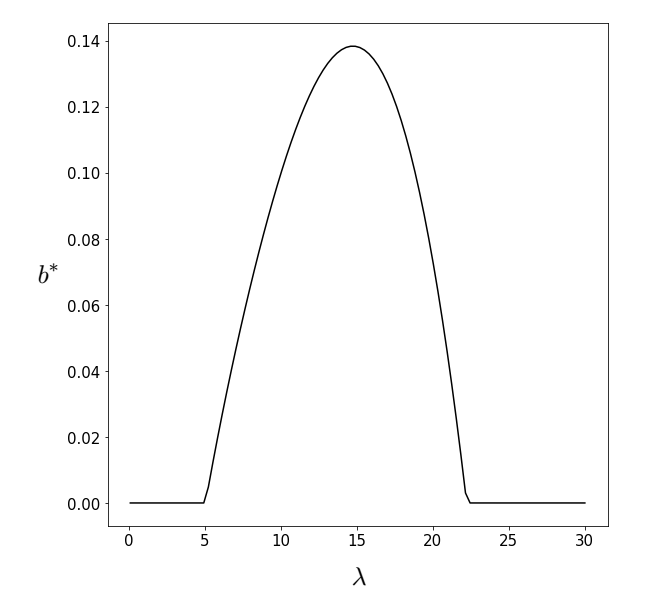}
\caption{Model and problem parameters: $c = 3.2, \alpha = 9.3, q = 1.6, p = 9.8, K = 1.9$.}
\label{fig:rateofclaim}
\end{figure}

\section*{Acknowledgements}

We thank two anonymous referees for their valuable comments and suggestions which have significantly improved the paper. We also thank Dante Mata L\'opez for fruitful discussions on the verification procedure.

Funding in support of this work was provided by a Discovery Grant (RGPIN-2019-06538) from the Natural Sciences and Engineering Research Council of Canada (NSERC) and a PhD Scholarship from the Fonds de Recherche du Québec - Nature et Technologies (FRQNT).

%
%
\bibliographystyle{abbrv}
\bibliography{references_de-finetti}

\appendix

\section{New identity}

The following is a generalization of Lemma 1 in \cite{renaud_2014}.
\begin{lem}\label{new-identity-lemma}
For all $x \geq b \geq 0$, we have
\begin{align}\label{new-identity-equation}
\esub{x}&\left[e^{-q \nu_b^-} Z_{q,p}\left(Y_{\nu_b^-}\right) \I{\nu_b^- < \infty} \right] \nonumber \\ 
& =p \int_0^\infty e^{-\Phi(p+q)y} w_b^{(q)}(x;-y) \mathrm{d}y -K \mathbb{W}^{(q)}(x-b) \int_0^\infty e^{-\Phi_K(q) z} Z_{q,p}^\prime(b+z) \mathrm{d}z.
\end{align}
\end{lem}
\begin{proof}
We have, using Fubini's theorem and the fact that $\set{Y_t + y : t \geq 0 \mid Y_0 = x}$ has the same law as $\set{Y_t : t \geq 0 \mid Y_0 = x + y}$, \begin{align}
\esub{x}&\left[e^{-q \nu_b^-} Z_{q,p}\left(Y_{\nu_b^-}\right) \I{\nu_b^- < \infty} \right] \nonumber \\
&= \esub{x}\left[e^{-q \nu_b^-} \left(p \int_0^\infty e^{-\Phi(p+q)y} W^{(q)}\left(Y_{\nu_b^-} + y\right) \mathrm{d}y\right) \I{\nu_b^- < \infty}\right] \nonumber \\
&= p \int_0^\infty e^{-\Phi(p+q)y} \esub{x+y}\left[e^{-q \nu_{b+y}^-} W^{(q)}\left(Y_{\nu_{b+y}^-}\right) \I{\nu_{b+y}^- < \infty}\right] \mathrm{d}y.
\label{substitute1real}
\end{align}

From Lemma 1 in \cite{renaud_2014}, it is known that, for $b \leq x \leq c$, we have \begin{equation*}
\esub{x}\left[e^{-q \nu_b^-} W^{(q)}\left(Y_{\nu_{b}^-}\right) \I{\nu_b^- < \nu_c^+}\right] = w_b^{(q)}(x;0) - \frac{\mathbb{W}^{(q)}(x-b)}{\mathbb{W}^{(q)}(c-b)} w_b^{(q)}(c;0).
\end{equation*}

By letting $c \to \infty$, using the monotone convergence theorem, we have \begin{align*}
\esub{x} & \left[e^{-q \nu_b^-} W^{(q)}\left(Y_{\nu_{b}^-}\right) \I{\nu_b^- < \infty} \right] = w_b^{(q)}(x;0) - K \mathbb{W}^{(q)}(x-b) \int_0^\infty e^{-\Phi_K(q) y} W^{(q) \prime} (z+b) \mathrm{d}z.
\end{align*}

Indeed, when $c > b$, we have \begin{align*}
&\frac{w_b^{(q)}(c;0)}{\mathbb{W}^{(q)}(c-b)} = \frac{W^{(q)}(c)}{\mathbb{W}^{(q)}(c-b)} + K \int_b^c \frac{\mathbb{W}^{(q)}(c-z)}{\mathbb{W}^{(q)}(c-b)} W^{(q) \prime}(z) \mathrm{d}z \\
&= \frac{W^{(q)}(c)}{\mathbb{W}^{(q)}(c-b)} + K \int_0^{c-b} \frac{\mathbb{W}^{(q)}(c-b-z)}{\mathbb{W}^{(q)}(c-b)} W^{(q) \prime}(z+b) \mathrm{d}z \\
&= \frac{e^{\Phi(q)c}W_{\Phi(q)}(c)}{e^{\Phi_K(q)(c-b)}\mathbb{W}_{\Phi_K(q)}(c-b)} + K \int_0^{c-b} e^{-\Phi_K(q)y} \frac{\mathbb{W}_{\Phi_K(q)}(c-b-z)}{\mathbb{W}_{\Phi_K(q)}(c-b)} W^{(q) \prime}(z+b) \mathrm{d}z,
\end{align*}
where $W_{\Phi(q)}$ and $\mathbb{W}_{\Phi_K(q)}$ represent respectively the $0$-scale functions for the SNLPs $\left(X, \mathbf{P}_x^{\Phi(q)}\right)$ and $\left(Y, \mathbf{P}_x^{\Phi_K(q)}\right)$ given by the Esscher transformation (see \eqref{esscher}). Substituting into \eqref{substitute1real}, we get, using Fubini's theorem, \begin{align*}
&\esub{x}\left[e^{-q \nu_b^-} Z_{q,p}\left(Y_{\nu_b^-}\right) \I{\nu_b^- < \infty} \right] \\
&= p \int_0^\infty e^{-\Phi(p+q)y} \left\{w_{b+y}^{(q)}(x+y;0) - K \mathbb{W}^{(q)}(x-b) \int_0^\infty e^{-\Phi_K(q)z} W^{(q) \prime}(b+y+z) \mathrm{d}z\right\} \mathrm{d}y \\
&= p \int_0^\infty e^{-\Phi(p+q)y} w_b^{(q)}(x;-y) \mathrm{d}y - pK \mathbb{W}^{(q)}(x-b) \\
&\times \int_0^\infty e^{-\Phi_K(q) z} \left(\int_0^\infty e^{-\Phi(p+q)y} W^{(q) \prime}((b+z)+y) \mathrm{d}y \right) \mathrm{d}z \\
&= p \int_0^\infty e^{-\Phi(p+q)y} w_b^{(q)}(x;-y) \mathrm{d}y -K \mathbb{W}^{(q)}(x-b) \int_0^\infty e^{-\Phi_K(q) z} Z_{q,p}^\prime(b+z) \mathrm{d}z,
\end{align*}
which is exactly what we want.
\end{proof}

\section{Proof of Proposition~\ref{hardest}}\label{proof-of-hardest-prop}

As discussed in Section~\ref{section:convexity}, under our assumption on the Lévy measure, we have that $Z_{q,p}^\prime$ is decreasing on $[0, b^*]$. The representation of $V_{b^*}$ given in Corollary~\ref{cor:vbetoile} yields the same property for $V_{b^*}^\prime$, i.e., it is decreasing on $[0, b^*]$. Recalling that $V_{b^*}^\prime(b^*) = 1$, we see that all is left to show is that $V_{b^*}^\prime$ is decreasing from $1$ for $x \geq b^*$.

For $x \geq b^*$, using Leibniz rule, we can write
\begin{equation}\label{debutpreuve}
V_{b^*}^\prime(x) = -K\mathbb{W}^{(q)}(x-b^*) + \frac{(1+K\mathbb{W}^{(q)}(0))Z_{q,p}^\prime(x)}{h_p(b^*)} + \frac{K\int_{b^*}^x \mathbb{W}^{(q) \prime}(x-y) Z_{q,p}^\prime(y) \mathrm{d}y}{h_p(b^*)}.
\end{equation}

Using~\eqref{Zqprimeutile} and Fubini's theorem, we have
\begin{multline*}
K\int_{b^*}^x \mathbb{W}^{(q) \prime}(x-y) Z_{q,p}^\prime(y) \mathrm{d}y = \int_0^{\infty} Kpe^{-\Phi(p+q)z}\left[\int_0^{x} \mathbb{W}^{(q) \prime}(x-y) W^{(q) \prime}(y+z)\mathrm{d}y \right. \\
- \left. \int_0^{b^*} \mathbb{W}^{(q) \prime}(x-y) W^{(q) \prime}(y+z)\mathrm{d}y \right] \mathrm{d}z.
\end{multline*}

In \cite{kyprianou-et-al_2012}, the following identity is obtained: for $a > 0$,
\[
K\int_0^a \mathbb{W}^{(q) \prime}(a-y) W^{(q) \prime}(y) \mathrm{d}y = (1 - KW^{(q)}(0)) \mathbb{W}^{(q)\prime}(a) - (1+K\mathbb{W}^{(q)}(0))W^{(q)\prime}(a) .
\]
Using this identity, we can write
\begin{multline*}
K \int_0^{x} \mathbb{W}^{(q) \prime}(x-y) W^{(q) \prime}(y+z)\mathrm{d}y = K \int_z^{x+z} \mathbb{W}^{(q) \prime}(x+z-y) W^{(q) \prime}(y)\mathrm{d}y \\
= (1-KW^{(q)}(0)) \mathbb{W}^{(q)\prime}(x+z) - (1+K\mathbb{W}^{(q)}(0))W^{(q)\prime}(x+z) - K\int_0^{z} \mathbb{W}^{(q) \prime}(x+z-y) W^{(q) \prime}(y)\mathrm{d}y .
\end{multline*}

Putting the pieces back in~\eqref{debutpreuve} and using Fubini's theorem, we get
\begin{multline*}
V_{b^*}^\prime(x) = -K\mathbb{W}^{(q)}(x-b^*) + \frac{p(1 - KW^{(q)}(0))}{h_p(b^*)} \int_0^\infty e^{-\Phi(p+q)z} \mathbb{W}^{(q) \prime}(x+z) \mathrm{d}z \\
- \frac{K \int_0^\infty \int_y^\infty pe^{-\Phi(p+q)z}\mathbb{W}^{(q) \prime}(x+z-y) W^{(q) \prime}(y) \mathrm{d}z \mathrm{d}y}{h_p(b^*)} \\
- \frac{K \int_0^{b^*} \mathbb{W}^{(q) \prime}(x-y) Z_{q,p}^\prime(y) \mathrm{d}y}{h_p(b^*)}.
\end{multline*}

Applying Lemma~\ref{scalemonotone} to $\mathbb{W}^{(q)}$, we have that there exists a completely monotone function $f$ such that
\[
\mathbb{W}^{(q)}(x) = \frac{1}{\psi_K^\prime(\Phi_K(q))} e^{\Phi_K(q) x} - f(x)
\]
and consequently there exists a constant $C$ such that
\begin{multline*}
V_{b^*}^\prime(x) = C e^{\Phi_K(q) x} + Kf(x-b^*) - \frac{(1 - KW^{(q)}(0))\int_0^{\infty} pe^{-\Phi(p+q)y}f'(x+y) \mathrm{d}y}{h_p(b^*)} \\
+ \frac{K \int_0^\infty \int_y^\infty pe^{-\Phi(p+q)z} f'(x+z-y) W^{(q) \prime}(y) \mathrm{d}z \mathrm{d}y}{h_p(b^*)} + \frac{K \int_0^{b^*} f'(x-y) Z_{q,p}^\prime(y) \mathrm{d}y}{h_p(b^*)}.
\end{multline*}

By Bernstein's theorem, there exists a Borel measure $\mu$ such that $f(x) = \int_0^\infty e^{-xt} \mu(\mathrm{d}t)$. Therefore, using Fubini's theorem and elementary algebraic manipulations, we can further write
\begin{equation}\label{bernsteinv}
V_{b^*}'(x) = C e^{\Phi_K(q) x} + \int_0^\infty e^{-xt} u_{b^*}(t) \mu(\mathrm{d}t) ,
\end{equation}
where
\[
u_{b}(t) = K e^{bt} - \frac{K \int_0^{b} te^{ty}Z_{q,p}^\prime(y) \mathrm{d}y}{h_p(b)} + \frac{A_bt}{\Phi(p+q) + t}, 
\]
with
\[
A_b = \frac{p(1 - KW^{(q)}(0)) - KZ_{q,p}^\prime(0+)}{h_p(b)} .
\]
Note that $u_b$ is at least twice differentiable.

In the definition of $u_b$, we must have $C = 0$. Indeed, if $C \neq 0$, then the representation in~\eqref{bernsteinv} yields that $V_{b^*}'(x) \to \infty$, as $x \to \infty$, which is a contradiction with the fact that $V_{b^*}(x) \leq K/q$ for all $x \in \R$ by the formulation of our control problem. In conclusion, we have obtained the following representation:
\begin{equation}\label{bernsteinvreal}
V_{b^*}'(x) = \int_0^\infty e^{-xt} u_{b^*}(t) \mu(\mathrm{d}t).
\end{equation}

The rest of the proof is split in two parts. First, if $b^* = 0$, then from~\eqref{condition-generale} in Proposition~\ref{specifications-sur-b^*}, we have $h_p(0) \geq Z_{q,p}'(0)$. Consequently, for all $t \geq 0$,
\begin{equation*}
u_0(t) = K\left(1 - \left(\frac{Z_{q,p}'(0)}{h_p(0)}\right)\left(\frac{t}{\Phi(p+q) + t}\right)\right) + \left(\frac{p(1 - KW^{(q)}(0))}{h_p(0)}\right)\left(\frac{t}{\Phi(p+q) + t}\right) \geq 0 .
\end{equation*}
Indeed, we know that $h_p(0)>0$ and, from~\eqref{Wen0}, it is known that either $W^{(q)}(0)=0$ or $K < c = \left(W^{(q)}(0)\right)^{-1}$, whether $X$ has paths of unbounded or bounded variation, with the inequality being a standing assumption in our problem. It follows that, for all $x > 0$,
\[
V_{0}''(x) = \int_0^\infty (-t) e^{-xt} u_0(t) \mu(\mathrm{d}t) \leq 0. 
\]
In other words, $V_{0}$ is concave on $(0,\infty)$.

Second, let us start by noting that $u_{b^*}$ is infinitely differentiable and such that $u_{b^*}(0) = K$. Now, if $b^* > 0$, then assume there exists $0 < t_0 \leq \infty$ such that $u_{b^*}(t) \geq 0$, for $0 \leq t < t_0$, and that $u_{b^*}(t) \leq 0$ for $t > t_0$.  Under this assumption, we can use~\eqref{bernsteinvreal} to deduce that, for all $x > b^*$, 
\begin{multline}\label{doubleinferieur}
V_{b^*}^{\prime \prime}(x) = \int_0^\infty -te^{-(x-b^*)t} e^{-b^* t}u_{b^*}(t) \mu (\mathrm{d}t) \\
\leq e^{-(x-b^*)k} \int_0^\infty -te^{-b^* t}u_{b^*}(t) \mu(\mathrm{d}t) = e^{-(x-b^*)k}V_{b^*}^{\prime \prime}(b^*+) ,
\end{multline}
where
\[
V_{b^*}^{\prime \prime}(b^*+) = \frac{(1+K\mathbb{W}^{(q)}(0)) Z_{q,p}^{\prime \prime}(b^*)}{h_p(b^*)} .
\]
As discussed at the beginning of the proof, $Z_{q,p}^\prime$ is decreasing on $[0, b^*]$ because it is decreasing on $[0, c^*]$ and $b^* \leq c^*$ by Proposition~\ref{caracteristiques}. Consequently, $Z_{q,p}''(b^*) \leq 0$ and thus $V_{b^*}''(b^*+) \leq 0$, and by~\eqref{doubleinferieur}, we have $V_{b^*}''(x) \leq 0$, for all $x > b^*$. 

The rest of the proof consists in proving that there exists $0 < t_0 \leq \infty$ such that $u_{b^*}(t) \geq 0$, for $0 \leq t < t_0$, and that $u_{b^*}(t) \leq 0$ for $t > t_0$.  Note that $u_b(0)=K$.

Computing the first and second derivatives of $u_{b^*}$, and then using the fact that $y \mapsto V_{b^*}^\prime (y) = Z_{q,p}^\prime(y)/Z_{q,p}^\prime(b^*)$ is decreasing on $[0, b^*]$ (for both derivatives), we deduce that, for all $t \geq 0$,
\[
u_{b^*}'(t) \leq A_{b^*} \frac{\mathrm{d}}{\mathrm{d}t} \left(\frac{t}{\Phi(p+q) + t}\right)
\]
and
\[
u_{b^*}''(t) \leq A_{b^*} \frac{\mathrm{d}^2}{\mathrm{d}t^2} \left(\frac{t}{\Phi(p+q) + t}\right).
\]

If $A_{b^*} < 0$, then $u_{b^*}'(t) < 0$, for all $t > 0$, and the existence of $t_0$ such that $u_{b^*}(t) \geq 0$, for $0 \leq t \leq t_0$, and that $u_{b^*}(t) \leq 0$ for $t \geq t_0$ is guaranteed. If $A_{b^*} \geq 0$, then $u_{b^*}''(t) \leq 0$, and, in particular, $u_{b^*}$ is concave, and there exists $0 \leq t' < \infty$ such that $u_{b^*}$ is increasing for $0 \leq t \leq t'$ and decreasing for $t \geq t'$. In particular, the existence of $t_0$ is also guaranteed in this case. In both cases, the proof is complete.
\end{document}